\documentclass{amsart}
\usepackage{amsmath}
\usepackage{amsthm}
\usepackage{amssymb}
\usepackage[alphabetic]{amsrefs}
\usepackage{graphicx}
\usepackage{mathrsfs}
\usepackage{pinlabel}
\usepackage{tikz-cd}
\usepackage{svg}
\usepackage{mathtools}
\usepackage{geometry}
\usepackage[utf8]{inputenc}
\usepackage{hyperref}
\usepackage{cleveref}
\usepackage[section]{placeins}
\usepackage{subcaption}
\usepackage{hyperref}

\newtheorem{theorem}{Theorem}
\newtheorem*{MainTheorem*}{Main Theorem}
\newtheorem{corollary}[theorem]{Corollary}
\newtheorem{proposition}[theorem]{Proposition}

\newtheorem{claim}{Claim}
\theoremstyle{definition}

\theoremstyle{definition}
\newtheorem*{definition*}{Definition}
\theoremstyle{definition}
\newtheorem{question}{Question}

\newtheoremstyle{named}%
    {}{}{\itshape}{}{\bfseries}{.}{.5em}{\thmnote{#3}}
\theoremstyle{named}
\newtheorem*{namedtheorem}{Theorem}

\theoremstyle{remark}
\newtheorem*{remark*}{Remark}
\newtheorem{remark}{Remark}

\theoremstyle{remark}

\newcommand{\R}{\mathbb{R}}

\newcommand{\Q}{\mathbb{Q}}

\newcommand{\SO}{\mathrm{SO}}
\newcommand{\Isom}{\mathrm{Isom}}
\newcommand{\stab}{\mathrm{stab}}
\newcommand{\wt}{\widetilde}
\newcommand{\CH}{\mathrm{CH}}

\newcommand{\sys}{\mathrm{sys}}

\title{Systolic lattice extensions of classical Schottky groups}
\author{Junzhi Huang}
\author{Matthew Zevenbergen}
\date{\today}

\begin{document}

\address{Department of Mathematics, Yale University}
\email{junzhi.huang@yale.edu}
\address{Department of Mathematics, Boston College}
\email{zevenber@bc.edu}

\begin{abstract}
    We produce lattice extensions of a dense family of classical Schottky subgroups of the isometry group of $d$-dimensional hyperbolic space. The extensions produced are said to be systolic, since all loxodromic elements with short translation length are conjugate into the Schottky groups. Various corollaries are obtained, in particular showing that for all $d\geq3$, the set of complex translation lengths realized by systoles of closed hyperbolic $d$-manifolds is dense inside the set of all possible complex translation lengths. We also consider complex translation lengths in arithmetic hyperbolic $d$-manifolds, and provide a new way to construct non-arithmetic lattices.
\end{abstract}

\maketitle

\section{Introduction}

For a discrete group of isometries of $d$-dimensional hyperbolic space, $G\leq\Isom^+(\mathbb{H}^d)$, it is natural to ask whether $G$ is contained in a lattice. The dimension $3$ version of this question was considered by Brooks \cite{Brooks}, where he showed that every faithful geometrically finite Kleinian group representation can be approximated by faithful representations whose image is contained in a lattice. We will be interested in a particularly strong type of lattice extension $\Gamma$ of $G$, in which some geometric properties of $G\backslash\mathbb{H}^d$ are inherited by $\Gamma\backslash\mathbb{H}^d$. We use $\ell(g)$ for the translation length of a loxodromic element $g\in\Isom^+(\mathbb{H}^d)$.\par

\begin{definition*}
    For $D>0$ and $G\leq\Isom^+(\mathbb{H}^d)$, a subgroup $\Gamma\leq\Isom^+(\mathbb{H}^d)$ is a \textbf{$D$-systolic lattice extension of $G$} if 
    \begin{enumerate}
        \item $\Gamma$ is a lattice,
        \item $G\leq\Gamma$, 
        \item every loxodromic $g\in \Gamma$ with $\ell(g)\leq D$ is conjugate in $\Gamma$ to an element of $G$.
    \end{enumerate}
\end{definition*}

 Hence, if $\Gamma$ is a $D$-systolic lattice extension of $G$, then all closed geodesics of length at most $D$ in $\Gamma\backslash\mathbb{H}^d$ descend from closed geodesics in $G\backslash\mathbb{H}^d$. \par 

Motivated by a desire to include explicit elements of $\Isom^+(\mathbb{H}^d)$ in lattices, we will focus on the family of classical Schottky subgroups of \[\SO'(q,k)\leq\SO'(q,\R)\cong\Isom^+(\mathbb{H}^d)\] where $k$ is  a totally real number field, $q$ is a quadratic form over $k$ in $d+1$ variables, and $(k,q)$ is admissible. We refer readers to Section \ref{sec:background} for background on quadratic forms and classical Schottky groups. The main theorem we prove in this paper is the following: \par

\begin{namedtheorem}[\hypertarget{Theorem-SchottkyExtension}{Main Theorem}]
    If $(k,q)$ is admissible and $D>0$, every classical Schottky subgroup $F\leq \SO'(q,k)$ admits a torsion-free $D$-systolic lattice extension. If $q$ is anisotropic, this extension can be chosen to be cocompact.
\end{namedtheorem}

The starting place for our techniques in this theorem is the ``inbreeding" construction of Agol \cite{Agol}, which he uses to construct closed hyperbolic $4$-manifolds with arbitrarily short closed geodesics. The \hyperlink{Theorem-SchottkyExtension}{Main Theorem} will be proven in Section \ref{sec:main-proof}.\par

\subsection{Density of Schottky groups with lattice extensions} For a first application, we define $\mathcal{S}_{d,m}$ to be the space of faithful representations $\rho:\mathbb{F}_m\rightarrow\Isom^+(\mathbb{H}^d)$ of the free group $\mathbb{F}_m:=\langle a_1,\dots,a_m\rangle$ such that $\{\rho(a_1),\dots,\rho(a_m)\}$ is a standard generating set for a classical Schottky group (see Section \ref{subsec:Schottky}). We equip $\mathcal{S}_{d,m}$ with the algebraic topology (see \cite{MatsuzakiTaniguchi},\cite{McMullen}). The following corollary of the \hyperlink{Theorem-SchottkyExtension}{Main Theorem} follows from the openness of $\mathcal{S}_{d,m}$ in $\text{Hom}(\mathbb{F}_m,\Isom^+(\mathbb{H}^d))$ and density of $\SO'(q,k)$ in $\SO'(q,\R)$, where some anisotropic $q$ as above is fixed.

\begin{corollary}
    \label{Corollary:ExtandableSchottky}
    Representations whose images admits cocompact extensions are dense in $\mathcal{S}_{d,m}$.
\end{corollary}

This corollary extends the theorem of Brooks \cite{Brooks} to all dimensions, for this family of representations. We note, however, that our techniques are completely different from those of Brooks. We also mention a related result of Bowen \cite{Bowen}, which shows that any free group representation $\rho:F\rightarrow\Isom^+(\mathbb{H}^d)$ can be perturbed to a virtual homomorphism into any given lattice. One cannot, however, use Bowen's result to construct approximations of any specific free groups with the approximations themselves extending to lattices.   \par

We also record an application to the space $\mathcal{D}_d$ of discrete torsion-free subgroups of $\Isom^+(\mathbb{H}^d)$, equipped with the Chabauty (or geometric) topology (see \cite[Ch. E]{BenedettiPetronio}). The $d=3$ version of this statement again follows from the result of Brooks (see also \cite{PurcellSouto}, \cite{FuchsPurcellStewart}, and \cite[Sec. 3]{Zevenbergen}). This corollary will be proven in Section \ref{sec:Chabautyapplications}.

\begin{corollary}
\label{Corollary:chabauty}
    For $d\geq2$, every classical Schottky subgroup $F\leq\Isom^+(\mathbb{H}^d)$ is a Chabauty limit of cocompact lattices in $\mathcal{D}_d$. In particular, if $d\geq 3$, then $\mathcal{D}_d$ is not locally connected at $F$.
\end{corollary}
    
\subsection{Complex systoles in lattices} The \hyperlink{Theorem-SchottkyExtension}{Main Theorem} also yields several interesting corollaries about systoles of hyperbolic manifolds. For a lattice $\Gamma\leq\Isom^+(\mathbb{H}^d)$, the systole $\sys(\Gamma\backslash\mathbb{H}^d)$ of the finite volume hyperbolic $d$-manifold $\Gamma\backslash\mathbb{H}^d$ is the length of the a shortest closed geodesic in $\Gamma\backslash\mathbb{H}^d$. Equivalently, \[\sys(\Gamma\backslash\mathbb{H}^d):=\min_{\gamma\in\Gamma\text{ loxodromic}}\ell(\gamma).\] It is classical that there are closed hyperbolic $2$ and $3$-manifolds with arbitrarily small systole, the dimension $2$ case following from Teichm\"{u}ller theory and the dimension $3$ case following from Thurston's hyperbolic Dehn surgery theorem (see \cite{ThurstonNotes}*{Thm. 5.8.2}). The existence of closed hyperbolic $4$-manifolds with arbitrarily small systole was established by Agol \cite{Agol}, and was generalized to all dimensions by results of Bergeron-Haglund-Wise \cite{BHW} and Belolipetsky-Thomson \cite{BelolipetskyThomson}. Recently, Douba and the first author \cite{DoubaHuang} used similar techniques to show that the set of systoles of closed hyperbolic $d$-manifolds is dense in the positive reals, for all $d\geq2$. The techniques used in this family of papers (\cite{Agol}, \cite{BHW}, \cite{BelolipetskyThomson}, \cite{DoubaHuang}) are key influences in what appears here. \par

In addition to the length $\ell(g)$, each loxodromic $g\in\Isom^+(\mathbb{H}^d)$ has an associated \textbf{holonomy} $h(g)\in C(\SO(d-1))$, where $C(\SO(d-1))$ is the set of conjugacy classes in $\SO(d-1)$, equipped with the quotient topology. The holonomy $h(g)$ is determined by parallel transport of vectors along the axis of $g$. The pair $(\ell(g),h(g))$ is called the \textbf{complex translation length of }$g$. Accordingly, a \textbf{complex systole} of a hyperbolic $d$-manifold $M=\Gamma\backslash\mathbb{H}^d$ is defined to be the complex translation length of a loxodromic element $g\in\Gamma$ with $\sys(M)=\ell(g)$. Note that $M$ may have multiple complex systoles if $M$ contains multiple shortest closed geodesics.\par

In each of the papers \cite{Agol}, \cite{BelolipetskyThomson}, \cite{DoubaHuang}, the closed geodesics constructed which realize the systole have trivial holonomy. The next two corollaries will extend these results, in particular constructing the first examples known to the authors of closed hyperbolic $d$-manifolds with arbitrarily short geodesics whose associated holonomy is non-trivial, for $d\geq 4$.\par

We start with the following, which will be proven in Section \ref{sec:systoleapplications}, and was suggested to the authors by Sami Douba.

\begin{corollary}\label{cor:irrational-holonomy}
    For any dimension $d\geq3$ and any $\varepsilon>0$, there exists a closed hyperbolic $d$-manifold that has a shortest closed geodesic with length less than $\varepsilon$ and infinite order holonomy.
\end{corollary}

The next corollary extends the density result of \cite{DoubaHuang} to complex systoles.

\begin{corollary}\label{cor:dense-complex-systole}
    For any dimension $d\geq3$, the set of complex systoles of closed hyperbolic $d$-manifolds is dense inside $\R^+\times C(\SO(d-1))$.
\end{corollary}

\begin{proof}
Given admissible $(k,q)$ with $q$ anisotropic, it suffices to prove that any loxodromic $g\in\SO'(q,k)$ can be realized as the systole of some cocompact lattice, by the density of $\SO'(q,k)$ in $\Isom^+(\mathbb{H}^d)$. To this end, we apply the \hyperlink{Theorem-SchottkyExtension}{Main Theorem} to the cyclic group $F=\langle g\rangle$. Then, the \hyperlink{Theorem-SchottkyExtension}{Main Theorem} provides an $\ell(g)$-systolic cocompact lattice extension $\Gamma$ of $F$, as desired.
\end{proof}

\begin{remark}
    For any $n>0$, one can apply similar techniques to rank-$n$ classical Schottky groups to produce a family of closed hyperbolic $d$-manifolds whose complex systoles are dense and whose kissing numbers are each $n$ (that is, each of them has $n$ closed geodesics realizing the systole).
\end{remark}

\subsection{Obstructing arithmeticity} Complex translation lengths realized by elements in arithmetic lattices of simplest type (see Section \ref{subsec:arithmetic} for the definition) are much more restricted than for general lattices, in contrast with Corollary \ref{cor:dense-complex-systole}. Recall that a Salem number $\lambda$ is a real algebraic integer greater than 1 which is Galois conjugate to $\lambda^{-1}$, and all of whose Galois conjugates different from $\lambda^{-1}$ lie on the unit circle. Work of Emery, Ratcliffe, and Tschantz \cite{Emery2015SalemNA} shows, among other things, that if $\Gamma$ is arithmetic of simplest type and $g\in\Gamma$ is loxodromic, then $e^{\ell(g)}$ is either a Salem number or a square root of one, with the latter only occuring in odd dimensions. The next proposition, to be proven in Section \ref{subsec:construction}, shows that the complex translation lengths of loxodromics in arithmetic lattices of simplest type are additionally restricted. Compare with a recent work of Chu-Murillo \cite{ChuMurillo} which shows that there are infinitely many commensurability classes of arithmetic lattices of simplest type containing a loxodromic element with translation length $\log(\lambda)$.

\begin{proposition}\label{prop:arithmetic}
For any Salem number $\lambda$ and dimension $d\geq3$, there is a finite set of possible holonomies of loxodromic elements with translation length $\log(\lambda)$ in arithmetic lattices of simplest type in $\mathrm{Isom}^+(\mathbb{H}^d)$.
\end{proposition}

\begin{remark}
Let $\mathcal{S}$ denote the set of Salem numbers. The Salem conjecture of Lehmer \cite{Lehmer} states that $\inf\mathcal{S}>1$. A strengthening of this conjecture due to Boyd \cite{Boyd} says that the set of accumulation points of $\mathcal{S}$ is exactly the set of Pisot numbers (see also \cite[p. 30-31]{SalemNo}). In particular, Proposition \ref{prop:arithmetic} implies that if the Boyd-Salem conjecture holds, then for each $d\geq2$, the set of all complex translation lengths realized by hyperbolic elements in arithmetic $d$-manifolds of simplest type is discrete, in stark contrast with the result in Corollary \ref{cor:dense-complex-systole} for general lattices.
\end{remark}
\par

Proposition \ref{prop:arithmetic} allows us to justify non-arithmeticity of some manifolds produced by the \hyperlink{Theorem-SchottkyExtension}{Main Theorem}. Note that the constructions in \cite{Agol}, \cite{BelolipetskyThomson}, \cite{DoubaHuang} as well as the \hyperlink{Theorem-SchottkyExtension}{Main Theorem} here are already known to produce non-arithmetic lattices, whose non-arithmeticity is guaranteed by short geodesics (see the remark following Conjecture 10.12 in \cite{Gelander04}). Our method in the following corollary, however, is not dependent on containing short geodesics.

\begin{corollary}\label{cor:non-arithmetic}
    For any Salem number $\lambda$ and dimension $d\geq3$, there exists a non-arithmetic hyperbolic $d$-manifold whose systole is $\log(\lambda)$.
\end{corollary}

The manifolds constructed in Corollary \ref{cor:non-arithmetic}, which will be proven in Section \ref{subsec:construction}, are quasi-arithmetic (see Remark \ref{Remark:Subgroup}).\par

In Section \ref{sec:background}, we will provide necessary background information. Section \ref{sec:main-proof} will be devoted to the proof of the \hyperlink{Theorem-SchottkyExtension}{Main Theorem}. Finally, Section \ref{sec:applications} will focus on the remaining proofs of the corollaries.


\section{Background and notations}\label{sec:background}

\subsection{Hyperboloid models for hyperbolic $d$-space}\label{subsec:hyperboloids}
Let $q_0(x_0,\dots,x_{d})=x_0^2+\dots +x_{d-1}^2-x_{d}^2$ be the standard Lorentzian quadratic form. The set
\[
\mathbb{H}^d:=\{x\in\R^{d+1}\;|\;q_0(x)=-1\;\;\text{and}\;\;x_{d}>0\}
\]
is the hyperboloid model for $d$-dimensional hyperbolic space. The orientation preserving isometry group $\Isom^+(\mathbb{H}^{d})$ can be identified with the group
\[
\SO'(d,1):=\{A\in\mathrm{SL}(d+1,\R)\;|\;A^tJA=J\;\; \text{and} \;\; A(\mathbb{H}^d)=\mathbb{H}^d\}
\]
where $J:=\mathrm{diag}(1,\dots,1,-1)$ is the coefficient matrix of $q_0$.

For a real quadratic form $q$ in $d+1$ variables of signature $(d,1)$, let $Q\in\mathrm{GL}(d+1,\R)$ be the associated symmetric coefficient matrix. For any subring $R$ of $\R$, the special orthogonal group of $q$ over $R$ is defined by
\[
\SO(q,R):=\{A\in\text{SL}(d+1,R)\:|\;A^tQA=Q\}.
\]
and we let $\SO'(q,R)$ be the index-2 subgroup of $\SO(q,R)$ preserving the two components of $\{x\in\R^{d+1}\;|\;q(x)=-1\}$.
Since $q$ has signature $(d,1)$, there exists $M\in\mathrm{GL}(d+1,\R)$ so that $q(Mx)=q_0(x)$ for all $x\in\R^{d+1}$, and we have $\SO'(q,R)=M\SO'(q_0,R)M^{-1}$. Since $\SO'(q_0,\R)=\SO'(d,1)$, we will identify $\SO'(q,\R)$ with $\Isom^+(\mathbb{H}^d)$.

\subsection{Arithmetic lattices of simplest type}\label{subsec:arithmetic}
Let $k$ be a totally real number field, and let $q$ be a quadratic form in $d+1$ variables with coefficients in $k$. Such a pair $(k,q)$ is \textbf{admissible} if $q$ has signature $(d,1)$ at the identity embedding of $k$, and $q^\sigma$ is positive definite for any other real embedding $\sigma$ of $k$. Let $\mathcal{O}_k$ be the ring of integers of $k$. By the Borel--Harish-Chandra theorem \cite{BorelHarishChandra}, if $(k,q)$ is admissible, then $\SO'(q,\mathcal{O}_k)$ is a lattice in $\Isom^+(\mathbb{H}^d)$, and it is cocompact if and only if $q$ is anisotropic (i.e. $q(x)\neq 0$ for all nonzero $x\in k^{d+1})$. Note that for admissible $(k,q)$, if $k\neq\Q$ then $q$ is always anisotropic; if $k=\Q$, then $q$ is isotropic for $d\geq4$ by Meyer's theorem \cite[pg. 43]{Serre}. A discrete subgroup $\Gamma\leq\Isom(\mathbb{H}^d)$ is said to be \textbf{arithmetic of simplest type} if a conjugate of $\Gamma$ is commensurable with $\SO'(q,\mathcal{O}_k)$ for some admissible $(k,q)$; if this conjugate is furthermore contained in $\SO'(q,k)$, then $\Gamma$ is a \textbf{classical} arithmetic lattice. A hyperbolic $d$-orbifold is said to be arithmetic of simplest type if it is the quotient by an arithmetic subgroup of simplest type.

\subsection{Classical Schottky groups}\label{subsec:Schottky} 
We say subsets $U,V\subseteq\mathbb{H}^d$ are \textbf{strongly disjoint} if their closures in $\mathbb{H}^d\cup S^{d-1}_\infty$ are disjoint. A subgroup $F\leq\Isom^+(\mathbb{H}^d)$ is called a \textbf{classical Schottky group} if $F$ has a finite generating set $g_1,\dots,g_m$ so that there is a collection of pairwise strongly disjoint closed hyperbolic half-spaces $\{A_{-1},A_1,\dots,A_{-m},A_m\}$ with $g_i(A_{-i})=\overline{\mathbb{H}^d - A_i}$ for each $i>0$. Such a generating set is called a \textbf{standard generating set} for $F$, and is always a free generating set by the ping-pong lemma.\par


\section{Proof of the Main Theorem} 
\label{sec:main-proof}

\subsection*{Notations} In what follows, $\Gamma_n$ will denote a discrete subgroup in $\Isom^+(\mathbb{H}^d)$ where $n$ is a positive integer, and $M_n$ will denote the quotient $\Gamma_n\backslash\mathbb{H}^d$. For a metric space $Y$, when there is a designated embedding to $M_n$, usually coming from a covering map or a totally geodesic embedding, we use $Y^{(n)}$ to denote its image in $M_n$. For a subset $U$ of a metric space $Y$ and $r>0$, we use $N_r(U)$ to denote the closure of the $r$-neighborhood of $U$ in $Y$.

\begin{proof}[Proof of the \hyperlink{Theorem-SchottkyExtension}{Main Theorem}]

The proof is directly inspired by the aforementioned work by Agol, and can be summarized as follows (see also Figure \ref{fig:main-proof} for a schematic picture of the construction). Set $\Gamma_1$ to be the arithmetic lattice $\SO'(q,\mathcal{O}_k)$. First, various subgroup separability properties are used in Claims \ref{claim:1} and \ref{claim:2} to produce a finite index subgroup $\Gamma_3\leq\Gamma_1$ satisfying desirable properties. In particular, if $F$ has rank $m$, then $M_3$ contains $m$ pairs of embedded isometric totally geodesic 2-sided hypersurfaces, and the isometry between the two hypersurfaces in each pair is induced by a generator in a standard generating set of $F$. Then we cut $M_3$ along these hypersurfaces and glue back with a different pairing, using the isometries. Algebraically, this corresponds to applying the Klein-Maskit combination theorems, and each generator in a standard generating set for $F$ is contained in the fundamental group of the glued manifold. Finally, we prove that this extension is $D$-systolic using the properties ensured by Claims \ref{claim:1} and \ref{claim:2}.

We now present the proof in full detail. As in the statement of the theorem, let $F\leq\SO'(q,k)$ be a classical Schottky group and fix $D>0$. Fix a standard generating set $\{g_1,\dots,g_m\}$ for $F$, and let $A_{-1}, A_1,\dots,A_{-m},A_m$ be the associated pairwise strongly disjoint closed hyperbolic half-spaces so that $g_i(A_{-i})=\overline{\mathbb{H}^d- A_i}$ for each positive $i\in I$, where $I:=\{\pm1,\dots,\pm m\}$ is the index set. Let $P_i$ be the boundary hyperplane of $A_i$. By perturbing each $A_i$ we can assume that all the half spaces are defined over $k$, i.e. $P_i=(v_i)^\perp$ for some $v_i\in k^{d+1}$.

Define \[X:=\bigcup_{\substack{i,j\in I\\i\neq j}}P_i\cap N_D(P_j)\] so that $X$ is the set of points in $\cup_iP_i$ that are within distance $D$ of a different hyperplane in $\cup_iP_i$. Let $\Lambda_F$ be the limit set of $F$ and let $\CH(\Lambda_F)$ denote the convex hull of $\Lambda_F$. The set $X$ is compact since the planes $P_i$ are strongly disjoint, so there exists $R>0$ such that $X$ is contained in $N_R\big(\CH(\Lambda_F)\big)$. The set
    \[
    \mathrm{FD}:=\mathbb{H}^d-\bigcup_{i\in I}\text{int}(A_i)
    \]
is a fundamental domain for the $F$-action on $\mathbb{H}^d$,
and the set
    \[
    C_1:=N_R\big(\CH(\Lambda_F)\big)\cap \mathrm{FD}
    \]
is a compact fundamental domain for the $F$-action on $N_R(\CH(\Lambda_F))$.

Recall that we have defined \[\Gamma_1:=\SO'(q,\mathcal{O}_k).\] Since $P_i$ is defined over $k$, the stabilizer subgroup $\stab_{\Gamma_1}(P_i)$ is a lattice in $\Isom(P_i)$ by the Borel--Harish-Chandra theorem \cite{BorelHarishChandra}. 
    
\begin{claim}\label{claim:1}
    There exists a finite index subgroup $\Gamma_2\leq\Gamma_1$ with the following properties:
    \begin{itemize}
    \item[(A1)] $\Gamma_2$ is torsion-free;
    \item[(A2)] any loxodromic element of $\Gamma_2$ has translation length $>D$;
    \item[(A3)] $C_1$ embeds into $M_1$ under the covering map $\mathbb{H}^d\rightarrow M_1$;
    \item[(A4)] for each $i\in I$, the subgroup $H_i:=\stab_{\Gamma_2}(P_i)$ preserves the two sides of $P_i$;
    \item[(A5)]  $H:=\langle H_{-1},H_1,\dots,H_{-m},H_m\rangle$ is geometrically finite and decomposes as $H_{-1}*H_1*\dots*H_{-m}*H_m$;
    \item[(A6)] for any $i,j\in I$, if an $H$ translation of $P_i$ is within distance $D$ of $P_j$, then the translation is in the $H_j$ orbit of $P_i$.
    \end{itemize}
\end{claim}

To construct $\Gamma_2$ satisfying (A5) and (A6), choose auxiliary pairwise disjoint closed hyperbolic half-spaces $B_i$ for $i\in I$ so that $A_i\subseteq B_i$ for all $i$, and the hyperplanes $Q_i:=\partial B_i$ and $P_i$ are strongly disjoint. For each $i$, fix an arbitrary $p_i\in P_i$. By the residual finiteness of $\Gamma_1$, there exists a finite index subgroup $\Gamma_2\leq\Gamma_1$ such that for $H_i:=\stab_{\Gamma_2}(P_i)$, the Dirichlet fundamental polyhedron $E_i$ about $p_i$ for $H_i$ contains $Q_i$ and $\cup_{j\neq i}N_D(P_j)$ in its interior. Hence, for each $i\neq j$, we have \begin{equation}
\label{Eqn-PingPong}
    \mathbb{H}^d-\text{int}(E_i)\subseteq B_i\subseteq\text{int}(E_j),
\end{equation}
so it follows from the Klein combination theorem \cite{Klein} that $\langle H_{-1},H_1,\dots,H_{-m},H_m\rangle$ is discrete, decomposes as $H_{-1}*H_1*\dots*H_{-m}*H_m$, and has a fundamental domain given by $\cap_i E_i$. Since each $H_i$ acts on $P_i$ with finite covolume, each $E_i$ has finitely many faces. Therefore, $\cap_i E_i$ has finitely many faces, confirming that $\langle H_1,\dots,H_n\rangle$ is geometrically finite (see \cite{BowditchGF} for a discussion of different notions of geometrically finiteness). This confirms that $\Gamma_2$ may be chosen to satisfy (A5). Additionally, one can use a ping-pong argument to confirm that $\Gamma_2$ satisfies (A6), using Equation \ref{Eqn-PingPong} and the choice that for all $i$, $\cup_{j\neq i}N_D(P_j)$ is contained in the interior of $E_i$. Up to passing to a further finite index subgroup, we can assume $\Gamma_2$ satisfies (A1) by Selberg's Lemma, (A2) and (A3) by residual finiteness, and (A4) by the fact that the index-two subgroup of $\stab_{\Gamma_1}(P_i)$ which preserves the two sides of $P_i$ is separable in $\Gamma_1$ (see \cite{LongTotGeod},\cite{BergeronSep},\cite[Lem. 1.8]{BHW}), completing the proof of Claim 1.

Since $g_i$ is contained in the commensurator $\text{Comm}(\Gamma_2)$, we observe that for each $i>0$ in $I$, the subgroups
 \[
 K_{-i}:=H_{-i}\cap g_i^{-1}H_ig_i\;\;\;\;\;\text{ and }\;\;\;\;\; K_i:=H_i\cap g_iH_{-i}g_i^{-1}=g_iK_{-i}g_i^{-1}
 \]
are of finite index in $H_{-i}$ and $H_i$, respectively. Additionally, it follows from property (A5) that the group $K:=\langle K_{-1},K_1,\dots,K_{-m},K_m\rangle$ is geometrically finite and decomposes as $K_{-1}*K_1*\dots*K_{-m}*K_m$. For each $i$ in $I$, define $\Sigma_i:=K_i\backslash P_i$. Each $\Sigma_i$ embeds into $K\backslash\mathbb{H}^d$, and the union of the $\Sigma_i$ lies in the boundary of $\mathrm{CC}(K\backslash\mathbb{H}^d)$, the convex core of $K\backslash\mathbb{H}^d$. By condition (A3) and convexity, the set $C_1$ embeds into $\mathrm{CC}(K\backslash\mathbb{H}^d)$ as a convex subset connecting all the $\Sigma_i$. Let the image be denoted by $C_1^{(K)}$.

Let $\varepsilon_d$ be the Margulis constant for dimension $d$, and fix $\varepsilon<\varepsilon_d$ such that the distance between an $\varepsilon$-thin cusp neighborhood and the boundary of an $\varepsilon_d$-thin cusp neighborhood in a hyperbolic $d$-manifold is greater than $D$ \cite[Prop. 1.4]{FuterPurcellSchleimer2019}. We define $\Sigma_i^{nc}$ to be the compact submanifold of $\Sigma_i$ obtained by removing all cuspidal components of the thin part $(\Sigma_i)_{<\varepsilon}$. Define a compact subset $C_2\subseteq K\backslash\mathbb{H}^d$ by
\[
C_2:=C_1^{(K)}\cup\big(\bigcup_{i\in I}N_D(\Sigma_i^{nc})\big)
\]
where $N_D(\Sigma_i^{nc})$ is the $D$-neighborhood of $\Sigma_i^{nc}$ in $K\backslash\mathbb{H}^d$.

\begin{claim}\label{claim:2}
There exists a finite index subgroup $\Gamma_3\leq\Gamma_2$ with the following properties:
\begin{itemize}
\item [(B1)] $\stab_{\Gamma_3}(P_i)=K_i$ for all $i\in I$;
\item [(B2)] $C_2$ embeds into $M_3$ under the covering map $K\backslash\mathbb{H}^d\rightarrow\Gamma_3\backslash\mathbb{H}^d$;
\item [(B3)] $\cup_{i\in I}\Sigma_i$ embeds into $M_3$ as a disjoint union of totally geodesic 2-sided hypersurfaces $\Sigma_i^{(3)}$.
\end{itemize}
\end{claim}

Since $K$ is geometrically finite, it is separable in $\Gamma_2$ by the result of Bergeron, Haglund, and Wise \cite{BHW} (see also \cite{Wise2021}*{Thm. 15.13} and \cite{BergeronWise}*{Thm. 6.2} for the non-cocompact case). Hence, a characterization of separability due to Scott \cite{Scott} implies that there exists a finite index subgroup $\Gamma_3\leq\Gamma_2$ containing $K$ so that $C_2$ embeds into $M_3$ under the covering map, satisfying (B2). In particular, the union of the boundaries of $\Sigma_i^{nc}$ embeds into $M_3$. Therefore, the disjoint union $\sqcup_i \Sigma_i$ also embeds under the covering map $K\backslash\mathbb{H}^d\rightarrow M_3$, by the product structure of the cusps of hyperbolic manifolds and the fact that each $\Sigma_i$ is totally geodesic. That each image $\Sigma_i^{(3)}\subseteq M_3$ is 2-sided folows from (A4). Therefore, conditions (B1) and (B3) are also satisfied, finishing the proof of Claim 2.

\begin{figure}[htb]
\centering
\labellist
\pinlabel $\mathbb{H}^d$ at 876 1342
\pinlabel $K\backslash\mathbb{H}^d$ at 443 1070
\pinlabel $C_1$ at 856 1627
\pinlabel $F\backslash\mathbb{H}^d$ at 1450 1356
\pinlabel $M_3$ at 260 43
\pinlabel $M'$ at 790 43
\pinlabel $M_6$ at 1433 43
\pinlabel cut at 486 609
\pinlabel glue at 1075 609
\pinlabel $I(\Sigma_i)$ at 668 991
\pinlabel $O(\Sigma_i)$ at 892 991
\endlabellist
\includegraphics[scale=.2]{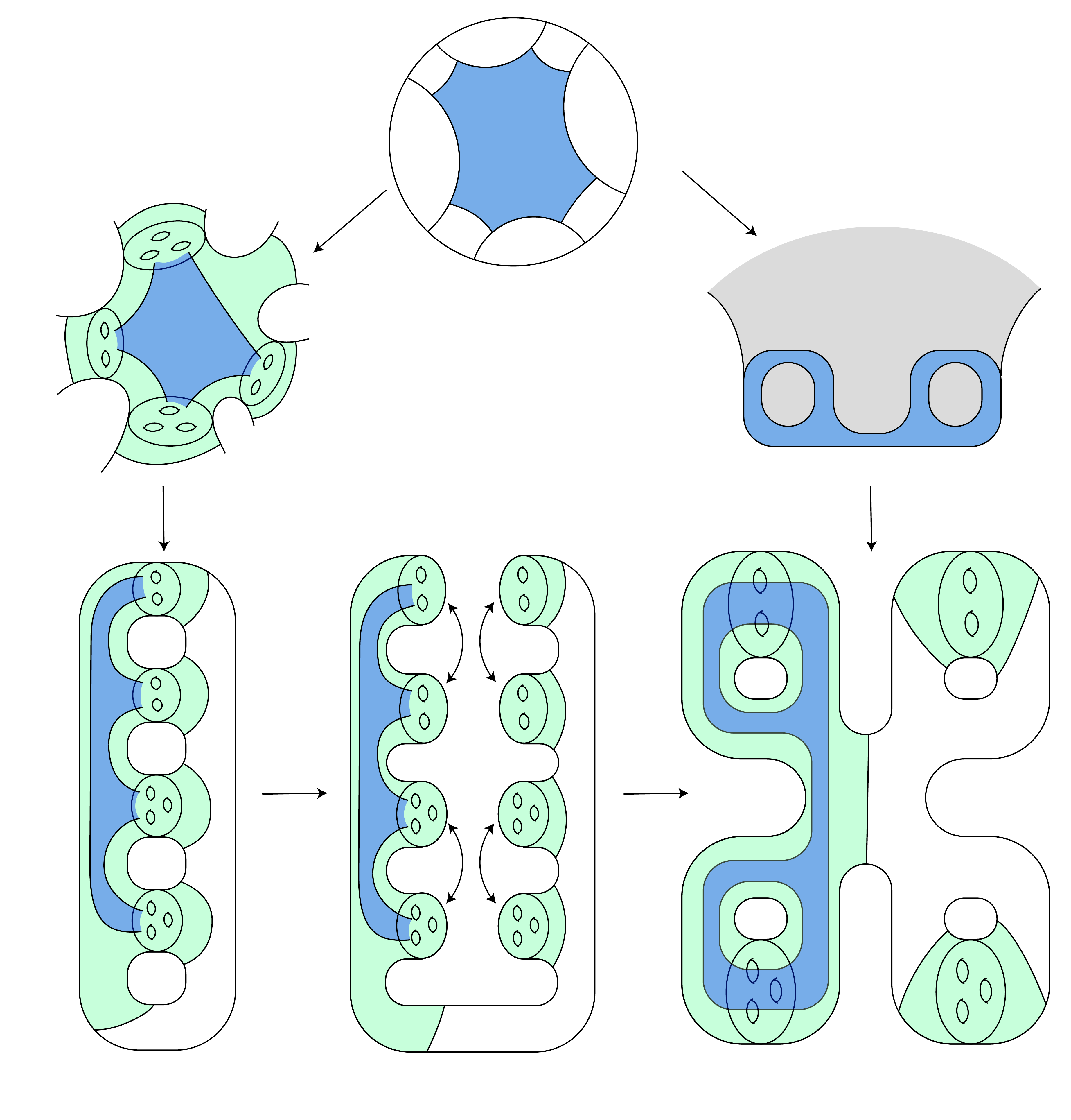}
\caption{The compact set $C_1$ and its images are shown in blue, and the compact set $C_2$ and its images are shown in green. Note that the image of $C_1$ in $F\backslash\mathbb{H}^d$ is $N_R\big(\mathrm{CC}(F\backslash\mathbb{H}^d)\big)$. All the arrows without labels are covering maps.}
\label{fig:main-proof}
\end{figure}
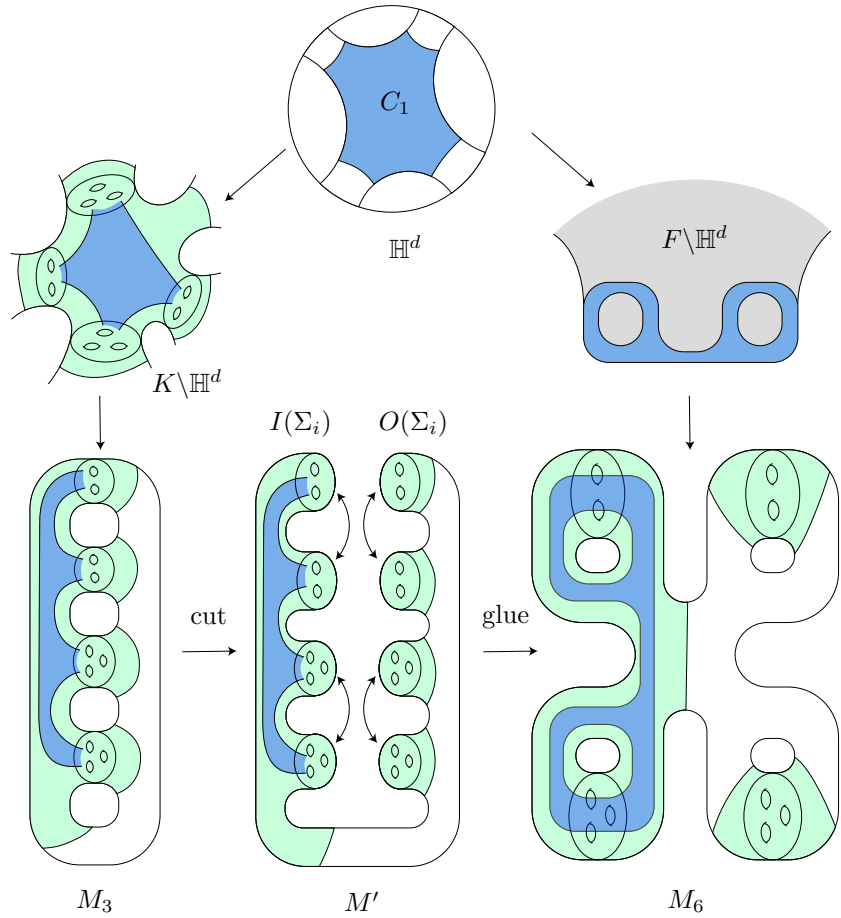

For depictions of the steps that follow, we refer the readers to Figure \ref{fig:main-proof}. Construct a finite volume (possibly disconnected) hyperbolic $d$-manifold $M'$ with totally geodesic boundary by slicing $M_3$ along each embedded hypersurface $\Sigma_i^{(3)}$. Let $M$ be the component of $M'$ containing the embedded image $C_1^{(3)}$ of $C_1$. Each $\Sigma_i^{(3)}$ gives rise to two components of $\partial M'$. We denote the component adjacent to $C_1^{(3)}$ by $I(\Sigma_i)$ and the other one by $O(\Sigma_i)$. The manifold $M$ corresponds to a subgroup $\Gamma_4\leq\Gamma_3$ isomorphic to $\pi_1(M)$. In fact, $M$ is the convex core of $M_4$.
Additionally, we have $\stab_{\Gamma_4}(P_i)=K_i$ for all $i$; since $K_i=g_iK_{-i}g_i^{-1}$ for each $i>0$, it then follows from the second Klein-Maskit combination theorem (see, e.g., \cite{AbikoffMaskit} and \cite{LiOhshikaWang}) that the group $\Gamma_5:=\langle\Gamma_4,g_1,\dots,g_m\rangle$ is discrete and torsion-free. Geometrically, $\mathrm{CC}(M_5)$ is obtained from $M$ by gluing $I(\Sigma_{-i})$ to (the isometric) $I(\Sigma_i)$ using the isometry descending from the action of $g_i$ on $\mathbb{H}^d$, for each $i>0$. 

While $\Gamma_5$ is a discrete extension of $F$, it need not be a lattice: it is possible that $M$ contains a component of $O(\Sigma_i)$, in which case $M_5$ has infinite volume. To obtain a finite volume manifold without boundary, we glue the remaining boundary components of $M_5$ with the boundaries of other components in $M'$. More precisely, we glue $O(\Sigma_{-i})$ and $O(\Sigma_i)$ by the composition
\[
O(\Sigma_{-i})\to I(\Sigma_{-i})\to I(\Sigma_{i})\to O(\Sigma_i)
\]
where the first and the third maps are canonical isometries coming from cutting, and the second map is the isometry induced by $g_i$. The resulting manifold is of finite-volume, orientable, and without boundary (but possibly disconnected). Let $M_6$ be the component containing $M$ and let $\Gamma_6$ be the Kleinian group containing $\Gamma_5$ associated to $M_6$. \par

It remains to prove that $\Gamma:=\Gamma_6$ is a $D$-systolic extension of $F$. To see this, take any closed geodesic $\alpha$ of length $\leq D$ in $M_6$. Since the systole of $\Gamma_3$ is greater than $D$ by (A2), $\alpha$ must intersect the hypersurfaces along which we did the cutting and gluing. Further, the choice of $\varepsilon$ implies that $\alpha$ cannot intersect any $\varepsilon$-thin cusp neighborhood of $M_6$. By undoing the gluing, $\alpha$ is cut into a union $\beta$ of properly immersed geodesic arcs in $M'$ with length at most $D$. As no component of $\beta$ intersects an $\varepsilon$-thin cusp neighborhood, the choice of $C_2$ and condition (B2) imply that the components of $\beta$ must have endpoints on $\cup I(\Sigma_i)$. By gluing $M'$ back to $M_3$, $\beta$ becomes a union of arcs $\beta_3\subseteq M_3$ with endpoints in $\cup(\Sigma_i^{nc})^{(3)}$ on the side of the image of $C_2$. Since we choose $C_2$ to contain a $D$-neighborhood of each $\Sigma_i^{nc}$, condition (B2) implies that any component $\gamma$ of $\beta_3$ must have a lift to $K\backslash\mathbb{H}^d$ with endpoints on $\cup\Sigma_i^{nc}$, which has a further lift $\wt{\gamma}$ to $\mathbb{H}^d$ with endpoints on $\cup P_i$, by (A6). But $\wt{\gamma}$ has length at most $D$, so our choice of $X$ and $C_1$ implies it must lie in $C_1$. The projection of $C_1$ into $M_6$ factors through the projection into $F\backslash\mathbb{H}^d$ as $N_R(\mathrm{CC}(F\backslash\mathbb{H}^d))$, and the latter embeds into $M_6$. Therefore, $\alpha$ is covered by a closed geodesic in $N_R(\mathrm{CC}(F\backslash\mathbb{H}^d))$, which means that $\alpha$ is conjugate to an element in $F$.

For the last cocompactness statement, by \cite{BorelHarishChandra} (see also Section \ref{subsec:arithmetic}) it suffices to observe that $\Gamma$ is cocompact if and only if $\Gamma_1$ is.
\end{proof}

\begin{remark}
\label{Remark:Subgroup}
    Although the algebraic description of gluing $O(\Sigma_i)$'s is not as explicit as gluing $I(\Sigma_i)$'s, it can be seen by interpreting each gluing in the construction in terms of the Klein-Maskit combination theorems that $\Gamma$ is contained in the group generated by $\Gamma_1$ and $F$. In particular, $\Gamma$ is a subgroup of $\SO'(q,k)$. Note that this immediately implies that $\Gamma$ is quasi-arithmetic in the sense of Vinberg (cf. \cite{Thomson2016}). 
\end{remark}

\begin{remark}
\label{Remark:ConvexCoreNeighborhood}
    In the proof of the \hyperlink{Theorem-SchottkyExtension}{Main Theorem}, we see that the $R$-neighborhood of the convex core $\mathrm{CC}(F\backslash\mathbb{H}^d)$ embeds into the finite volume manifold $\Gamma\backslash\mathbb{H}^d$ under the covering map. Notice that we are free to choose $R$ as large as we would like in the process of producing an extension.
\end{remark}

\begin{remark}
\label{Remark:Cubulation}
    When $\SO'(q,\mathcal{O}_k)$ is cocompact, the lattice $\Gamma$ we construct is cubulable, i.e. $\Gamma$ admits a proper cocompact action on a CAT(0) cube complex. This is because $\SO'(q,\mathcal{O}_k)$ is hyperbolic and cubulable by \cite{BHW}, so the quasiconvex subgroups of $\SO'(q,\mathcal{O}_k)$ are also cubulable by \cite{HuglandConvex}. The remark now follows from Agol's virtually special theorem \cite{AgolVHC}, Wise's quasiconvex hierarchy theorem \cite{Wise2021} (see also \cite{AgolGrovesManning}) and the fact that $\Gamma$ is obtained as combinations of $M'$ along quasiconvex subgroups.
\end{remark}


\section{Applications}
\label{sec:applications}

\subsection{An application to the Chabauty topology}

\label{sec:Chabautyapplications}

Before proving Corollary \ref{Corollary:chabauty}, we briefly recall that a sequence $\{\Gamma_n\}$ of discrete torsion-free subgroups of $\Isom^+(\mathbb{H}^d)$ converges to $\Gamma$ in the Chabauty topology on $\mathcal{D}_d$ if and only if 
\begin{enumerate}
    \item every accumulation point $g\in\Isom^+(\mathbb{H}^d)$ of a sequence $g_n\in\Gamma_n$ is contained in $\Gamma$, and
    \item for all $g\in \Gamma$, there exists a sequence $g_n\in\Gamma_n$ so that $g_n\rightarrow g$ in $\Isom^+(\mathbb{H}^d)$.
\end{enumerate}
We refer the readers to \cite[Ch. E]{BenedettiPetronio} for background on the Chabauty topology.

\begin{proof}[Proof of Corollary \ref{Corollary:chabauty}]
    Fix an admissible pair $(k,q)$ with anisotropic $q$ of signature $(d,1)$. Identifying $\Isom^+(\mathbb{H}^d)$ with $\SO'(q,\R)$, choose $\rho\in\mathcal{S}_{d,m}$ whose image is $F$, where $m$ is the rank of $F$. As is done in Corollary \ref{Corollary:ExtandableSchottky}, we may choose a sequence $\{\rho_n\}\subseteq\mathcal{S}_{d,m}$ which algebraically converges to $\rho$ so that each $\rho_n(\mathbb{F}_m)\leq\SO'(q,k)$. Set $F_n:=\rho_n(\mathbb{F}_m)$ and note that $F_n\rightarrow F$ in $\mathcal{D}_d$, since the algebraic topology on $\mathcal{S}_{d,m}$ is equivalent to the strong topology \cite[Cor. 4.3]{McMullen}. For each $n$, the \hyperlink{Theorem-SchottkyExtension}{Main Theorem} and Remark \ref{Remark:ConvexCoreNeighborhood} provide a cocompact lattice extension $\Gamma_n$ of $F_n$ so that if $h\in\Gamma_n$ satisfies \[h\cdot N_n(\mathrm{CH}(\Lambda_{F_n}))\cap N_n(\mathrm{CH}(\Lambda_{F_n}))\neq\varnothing,\] then $h\in F_n$.\par

    Since $\rho_n\rightarrow\rho$ strongly, it follows from a result of McMullen \cite{McMullen} that $\Lambda_{F_n}\rightarrow\Lambda_F$ in the Hausdorff topology; a result of Bowditch \cite{BowditchCHs} on convex hulls then implies that $\mathrm{CH}(\Lambda_{F_n})\rightarrow\mathrm{CH}(\Lambda_F)$ in the Chabauty topology on the set of closed subsets of $\mathbb{H}^d$. It follows that there exists $p\in\mathbb{H}^d$ so that $p\in N_1(\mathrm{CH}(F_n))$ for $n$ sufficiently large.\par

    Suppose that after passing to a subsequence, there exist elements $h_n\in\Gamma_n$ and $h\in\SO'(q,\R)$ so that $h_n\rightarrow h$ in $\SO'(q,\R)$. To confirm that $\Gamma_n\rightarrow F$ in $\mathcal{D}_d$, it suffices to show that $h\in F$, which will confirm the first condition for convergence in the Chabauty topology discussed above (the second condition being immediate, since $F_n\leq\Gamma_n$ for all $n$). Since $h_n\cdot p\rightarrow h\cdot p$ in $\mathbb{H}^d$, we see that $h_n\cdot p\in N_n(\mathrm{CH}(\Lambda_{F_n}))$ for $n$ sufficiently large. Thus, it follows from the choice of the extensions $\Gamma_n$ that $h_n\in F_n$ for $n$ sufficiently large. Since $F_n\rightarrow F$ in $\mathcal{D}_n$, we conclude that $h\in F$, as desired.\par

    For the last statement of the corollary, note that if $d\geq3$, then for each $n$, \[\{g\Gamma_ng^{-1}\;|\;g\in\SO'(q,\R)\}\] is a connected component of $\mathcal{D}_d$ (see \cite[Thm. E.2.4]{BenedettiPetronio} and \cite[Thm. A]{Zevenbergen}). Therefore, we have shown that $F$ is approximated in $\mathcal{D}_d$ by elements which lie in different connected components, so we see that $\mathcal{D}_d$ is not locally connected at $F$.
\end{proof}


\subsection{Applications to complex translation lengths in lattices}

\label{sec:systoleapplications}

We start by providing a proof for Corollary \ref{cor:irrational-holonomy}, in which we construct closed hyperbolic $3$-manifolds containing an arbitrarily short geodesic whose holonomy has infinite order.

\begin{proof}[Proof of Corollary \ref{cor:irrational-holonomy}]
    Fix $\varepsilon>0$ and $d\geq3$. Let $k=\Q(\sqrt2)$ and $q(x_0,\dots,x_d)=x_0^2+\dots+x_{d-1}^d-\sqrt{2}x_d^2$, and note that $(k,q)$ is admissible, with $q$ anisotropic. We will construct a loxodromic isometry $g\in\SO'(q,k)$ so that the complex translation length of $g$ is $(\ell(g),[A])$, where $\ell(g)<\varepsilon$ and $A\in\SO(d-1)$ has infinite order. Then, we may apply the \hyperlink{Theorem-SchottkyExtension}{Main Theorem} to construct an $\varepsilon$-systolic lattice extension of $\langle g\rangle$, which will establish the corollary.\par

    To construct $g$, we first choose $A\in\SO(d-1)$ with infinite order and coefficients in $k$. For example, we may let 
    \[R=\begin{pmatrix}
        3/5 & -4/5\\
        4/5  & 3/5
    \end{pmatrix}
    \hspace{.4in}\text{and} \hspace{.4in}
    A=\begin{pmatrix}
        R & \\
          & I_{d-3}
    \end{pmatrix}\]
    where $I_{n}$ is the $n\times n$ identity matrix (which is empty if $n=0$). Set $q'(x,y)=x^2-\sqrt{2}y^2$. By the density of $\SO'(q',k)$ in $\SO'(q',\R)=\Isom^+(\mathbb{H}^1)$, we may choose $T\in\SO'(q',k)$ with translation length $\ell(T)$ less than $\varepsilon$.\par

    Now, we define 
    \begin{equation*}
    \label{eqn:InfiniteOrderHolonomy}
        g:=\begin{pmatrix}
        A& \\
         & T
    \end{pmatrix}=
    \begin{pmatrix}
        A & \\
        & I_{2}
    \end{pmatrix}\cdot
    \begin{pmatrix}
        I_{d-1} & \\
                & T
    \end{pmatrix}
    \end{equation*}
    and observe that $g\in\SO'(q,k)$ is loxodromic with complex translation length $(\ell(T),[A])$. 
\end{proof}

\subsection{A construction of non-arithmetic lattices}\label{subsec:construction}

We will first prove Proposition \ref{prop:arithmetic}, establishing the finiteness of possible holonomies of loxodromics which have a fixed translation length $\log(\lambda)$, where $\lambda$ is a Salem number, in arithmetic $d$-manifolds of simplest type.

\begin{proof}[Proof of Proposition \ref{prop:arithmetic}]

Recall that $\lambda>1$ is a Salem number and $d\geq 3$. Let $q$ be an admissible quadratic form on $d+1$ variables defined over a totally real number field $k$, and let $\Gamma$ be a lattice commensurable to $\SO'(q,\mathcal{O}_k)$ containing a hyperbolic element $W$ with $\ell(W)=\log(\lambda)$. By Lemmas 4.2 and 4.5 of \cite{Emery2015SalemNA}, $W^2$ is contained in a classical arithmetic lattice in $\SO'(q,k)$. Since powers of a Salem number are Salem numbers and each loxodromic isometry has only finitely many square roots in $\Isom^+(\mathbb{H}^d)$, it suffices to assume $\Gamma$ is classical (see Section \ref{subsec:arithmetic}) and contained in $\SO'(q,k)$. \par

Let $p(x)$ be the characteristic polynomial of $W$, and note that $p(x)$ has coefficients in $\mathcal{O}_k$, by \cite[Lem. 3.1]{Emery2015SalemNA}. Since $W$ is hyperbolic with translation length $\log(\lambda)$, the roots of $p(x)$ consist of the simple real roots $\lambda$ and $\lambda^{-1}$, pairs of conjugate complex roots on the unit circle encoding the holonomy of $W$, and possibly simple or multiple roots $\pm1$. By \cite[Thm. 5.2]{Emery2015SalemNA}, $k$ is a subfield of $\Q(\lambda+\lambda^{-1})$. Since there are only finitely many such $k$ for each $\lambda$, it suffices to show that there are only finitely many possibilities for the polynomial $p(x)$, for fixed $\lambda$, $d$, and $k$.\par

Let $p_1(x)\in\mathcal{O}_k[x]$ be the minimal polynomial of $\lambda$ over $k$, and define $p_2(x)$ to be $p(x)/p_1(x)\in\mathcal{O}_k[x]$. The proof of \cite[Thm. 5.2]{Emery2015SalemNA} shows that $\lambda^{-1}$ is also a root of $p_1(x)$, so all roots of $p_2(x)$ have absolute value 1. Moreover, for any embedding $\sigma:k\to\R$ that is not the identity embedding, $W^\sigma$ lies in the compact group $\SO(q^\sigma,\R)$. The polynomial $p_2^{\sigma}(x)$ divides the characteristic polynomial of $W^\sigma$, so all roots of $p_2^{\sigma}(x)$ have absolute value $1$ as well. This shows that the coefficients of $p_2(x)$ are elements of $\mathcal{O}_k$ whose Galois conjugates all have absolute value bounded in terms of $d$; hence, there is only a finite number of choices for these coefficients. In turn, there are only finitely many possibilities for $p(x)=p_1(x)p_2(x)$, finishing the proof of the proposition.
\end{proof}

We now provide a proof of Corollary \ref{cor:non-arithmetic}, which combines Proposition \ref{prop:arithmetic} and the \hyperlink{Theorem-SchottkyExtension}{Main Theorem} to construct non-arithmetic hyperbolic $d$-manifolds with systole $\log(\lambda)$, for any Salem number $\lambda$ and $d\geq3$.

\begin{proof}[Proof of Corollary \ref{cor:non-arithmetic}]

Set $\mu=\lambda+\lambda^{-1}$ and observe that the quadratic form $q$ defined by the symmetric matrix 
\[Q=
\begin{pmatrix}
    I_{d-1}   &    & \\
    &        1&    \mu/2\\
    &    \mu/2&    1
\end{pmatrix}
\]
is admissible over $k:=\Q(\mu)$. For any matrix $A\in\SO(d-1)$ with rational entries, the matrix 
\[\gamma_A:=\begin{pmatrix}
        A & &\\
        &\mu & 1\\
        & -1 & 0
    \end{pmatrix}
\]
is loxodromic in $\SO'(q,k)$ with complex translation length $(\log(\lambda),[A])$. Since there are infinitely many conjugacy classes of rational matrices in $\SO(d-1)$, Proposition \ref{prop:arithmetic} allows us to choose $A_0\in\SO(d-1)$ so that $\gamma_{A_0}$ cannot be contained in an arithmetic lattice of simplest type.\par

Our \hyperlink{Theorem-SchottkyExtension}{Main Theorem} produces a torsion-free lattice $\Gamma$ containing $\gamma_{A_0}$ such that $\sys(\Gamma\backslash\mathbb{H}^d)=\log(\lambda)$. We claim that the $\Gamma$ produced by our construction is non-arithmetic. Indeed, if $\Gamma$ is arithmetic, the construction and commensurability criterion of Gromov--Piatetski-Shapiro \cite{GromovPiatetski-Shapiro} implies that $\Gamma$ must be commensurable to $\SO'(q,\mathcal{O}_k)$, hence arithmetic of simplest type, contradicting the choice of $\gamma_{A_0}$.
\end{proof}

It is unclear to the authors at this time whether Corollary \ref{cor:non-arithmetic} (or more generally the \hyperlink{Theorem-SchottkyExtension}{Main Theorem}) produces lattices that are not commensurable to lattices arising from the doubling arguments, referred to by Agol as \textit{inbreeding} in \cite{Agol}. To make this precise, we say a hyperbolic $d$-manifold $M$ is \textbf{inbred} if there exists a torsion-free discrete subgroup $\Gamma$ of an arithmetic lattice of simplest type such that $N:=\mathrm{CC}(\Gamma\backslash\mathbb{H}^d)$ is finite volume with non-empty totally geodesic boundary, and $M$ is isometric to the result of doubling $N$ along its boundary. We now end with the following question.

\begin{question}
    Are there manifolds produced by the \hyperlink{Theorem-SchottkyExtension}{Main Theorem} which are not commensurable to any inbred hyperbolic manifold?
\end{question}

\section*{Acknowledgements}

The authors thank Ian Biringer, Tamunonye Cheetham-West, Sami Douba, Dongryul M. Kim, Yair Minsky, and Eduardo Reyes for helpful conversations, comments, and feedback. The first author is partially supported by NSF grant DMS-2005328. This project started during a visit of the first author to Boston College. The first author would like to thank the Boston College Department of Mathematics for their hospitality.

\bibliographystyle{plain}
\bibliography{refs}

\begin{bibdiv}
\begin{biblist}

\bib{AgolGrovesManning}{article}{
      author={Agol, Ian},
      author={Groves, Daniel},
      author={Manning, Jason~Fox},
       title={An alternate proof of {W}ise's malnormal special quotient
  theorem},
        date={2016},
        ISSN={2050-5086},
     journal={Forum Math. Pi},
      volume={4},
       pages={e1, 54},
         url={https://doi.org/10.1017/fmp.2015.8},
      review={\MR{3456181}},
}

\bib{Agol}{article}{
      author={Agol, Ian},
       title={Systoles of hyperbolic 4-manifolds},
        date={2006},
     journal={arXiv preprint},
      eprint={arXiv:math/0612290},
         url={https://arxiv.org/abs/math/0612290},
}

\bib{AgolVHC}{article}{
      author={Agol, Ian},
       title={The virtual {H}aken conjecture},
        date={2013},
        ISSN={1431-0635,1431-0643},
     journal={Doc. Math.},
      volume={18},
       pages={1045\ndash 1087},
         url={https://elibm.org/article/10000267},
        note={With an appendix by Agol, Daniel Groves, and Jason Manning},
      review={\MR{3104553}},
}

\bib{AbikoffMaskit}{article}{
      author={Abikoff, William},
      author={Maskit, Bernard},
       title={Geometric decompositions of {K}leinian groups},
        date={1977},
        ISSN={0002-9327,1080-6377},
     journal={Amer. J. Math.},
      volume={99},
      number={4},
       pages={687\ndash 697},
         url={https://doi.org/10.2307/2373860},
      review={\MR{480992}},
}

\bib{BergeronSep}{article}{
      author={Bergeron, Nicolas},
       title={Premier nombre de {B}etti et spectre du laplacien de certaines
  vari\'et\'es hyperboliques},
        date={2000},
        ISSN={0013-8584},
     journal={Enseign. Math. (2)},
      volume={46},
      number={1-2},
       pages={109\ndash 137},
      review={\MR{1769939}},
}

\bib{BorelHarishChandra}{article}{
      author={Borel, Armand},
      author={Harish-Chandra},
       title={Arithmetic subgroups of algebraic groups},
        date={1962},
        ISSN={0003-486X},
     journal={Ann. of Math. (2)},
      volume={75},
       pages={485\ndash 535},
         url={https://doi.org/10.2307/1970210},
      review={\MR{147566}},
}

\bib{BHW}{article}{
      author={Bergeron, Nicolas},
      author={Haglund, Fr\'ed\'eric},
      author={Wise, Daniel~T.},
       title={Hyperplane sections in arithmetic hyperbolic manifolds},
        date={2011},
        ISSN={0024-6107,1469-7750},
     journal={J. Lond. Math. Soc. (2)},
      volume={83},
      number={2},
       pages={431\ndash 448},
         url={https://doi.org/10.1112/jlms/jdq082},
      review={\MR{2776645}},
}

\bib{Bowen}{article}{
      author={Bowen, Lewis},
       title={Free groups in lattices},
        date={2009},
        ISSN={1465-3060,1364-0380},
     journal={Geom. Topol.},
      volume={13},
      number={5},
       pages={3021\ndash 3054},
         url={https://doi.org/10.2140/gt.2009.13.3021},
      review={\MR{2546620}},
}

\bib{BowditchGF}{article}{
      author={Bowditch, B.~H.},
       title={Geometrical finiteness for hyperbolic groups},
        date={1993},
        ISSN={0022-1236,1096-0783},
     journal={J. Funct. Anal.},
      volume={113},
      number={2},
       pages={245\ndash 317},
         url={https://doi.org/10.1006/jfan.1993.1052},
      review={\MR{1218098}},
}

\bib{BowditchCHs}{article}{
      author={Bowditch, B.~H.},
       title={Some results on the geometry of convex hulls in manifolds of
  pinched negative curvature},
        date={1994},
        ISSN={0010-2571,1420-8946},
     journal={Comment. Math. Helv.},
      volume={69},
      number={1},
       pages={49\ndash 81},
         url={https://doi.org/10.1007/BF02564474},
      review={\MR{1259606}},
}

\bib{Boyd}{article}{
      author={Boyd, David~W.},
       title={Small {S}alem numbers},
        date={1977},
        ISSN={0012-7094,1547-7398},
     journal={Duke Math. J.},
      volume={44},
      number={2},
       pages={315\ndash 328},
         url={http://projecteuclid.org/euclid.dmj/1077312234},
      review={\MR{453692}},
}

\bib{BenedettiPetronio}{book}{
      author={Benedetti, Riccardo},
      author={Petronio, Carlo},
       title={Lectures on hyperbolic geometry},
      series={Universitext},
   publisher={Springer-Verlag, Berlin},
        date={1992},
        ISBN={3-540-55534-X},
         url={https://doi.org/10.1007/978-3-642-58158-8},
      review={\MR{1219310}},
}

\bib{Brooks}{article}{
      author={Brooks, Robert},
       title={Circle packings and co-compact extensions of {K}leinian groups},
        date={1986},
        ISSN={0020-9910,1432-1297},
     journal={Invent. Math.},
      volume={86},
      number={3},
       pages={461\ndash 469},
         url={https://doi.org/10.1007/BF01389263},
      review={\MR{860677}},
}

\bib{BelolipetskyThomson}{article}{
      author={Belolipetsky, Mikhail~V.},
      author={Thomson, Scott~A.},
       title={Systoles of hyperbolic manifolds},
        date={2011},
        ISSN={1472-2747,1472-2739},
     journal={Algebr. Geom. Topol.},
      volume={11},
      number={3},
       pages={1455\ndash 1469},
         url={https://doi.org/10.2140/agt.2011.11.1455},
      review={\MR{2821431}},
}

\bib{BergeronWise}{article}{
      author={Bergeron, Nicolas},
      author={Wise, Daniel~T.},
       title={A boundary criterion for cubulation},
        date={2012},
        ISSN={0002-9327,1080-6377},
     journal={Amer. J. Math.},
      volume={134},
      number={3},
       pages={843\ndash 859},
         url={https://doi.org/10.1353/ajm.2012.0020},
      review={\MR{2931226}},
}

\bib{ChuMurillo}{article}{
      author={Chu, Michelle},
      author={Murillo, Plinio G.~P.},
       title={Salem numbers and commensurability classes of arithmetic
  hyperbolic manifolds},
        date={2025},
     journal={arXiv preprint},
      eprint={arXiv:2506.20552},
         url={https://arxiv.org/abs/2506.20552},
}

\bib{DoubaHuang}{article}{
      author={Douba, Sami},
      author={Huang, Junzhi},
       title={Density of systoles of hyperbolic manifolds},
        date={2024},
        ISSN={1631-073X,1778-3569},
     journal={C. R. Math. Acad. Sci. Paris},
      volume={362},
       pages={1819\ndash 1824},
         url={https://doi.org/10.5802/crmath.689},
      review={\MR{4834584}},
}

\bib{Emery2015SalemNA}{article}{
      author={Emery, Vincent},
      author={Ratcliffe, John~G.},
      author={Tschantz, Steven~T.},
       title={Salem numbers and arithmetic hyperbolic groups},
        date={2019},
        ISSN={0002-9947,1088-6850},
     journal={Trans. Amer. Math. Soc.},
      volume={372},
      number={1},
       pages={329\ndash 355},
         url={https://doi.org/10.1090/tran/7655},
      review={\MR{3968771}},
}

\bib{FuterPurcellSchleimer2019}{article}{
      author={Futer, David},
      author={Purcell, Jessica~S.},
      author={Schleimer, Saul},
       title={Effective distance between nested {M}argulis tubes},
        date={2019},
        ISSN={0002-9947},
     journal={Trans. Amer. Math. Soc.},
      volume={372},
      number={6},
       pages={4211\ndash 4237},
         url={https://doi.org/10.1090/tran/7678},
      review={\MR{4009389}},
}

\bib{FuchsPurcellStewart}{article}{
      author={Fuchs, Urs},
      author={Purcell, Jessica~S.},
      author={Stewart, John},
       title={Constructing knots with specified geometric limits},
        date={2023},
        ISSN={0030-8730,1945-5844},
     journal={Pacific J. Math.},
      volume={324},
      number={1},
       pages={111\ndash 142},
         url={https://doi.org/10.2140/pjm.2023.324.111},
      review={\MR{4604668}},
}

\bib{Gelander04}{article}{
      author={Gelander, Tsachik},
       title={Homotopy type and volume of locally symmetric manifolds},
        date={2004},
        ISSN={0012-7094,1547-7398},
     journal={Duke Math. J.},
      volume={124},
      number={3},
       pages={459\ndash 515},
         url={https://doi.org/10.1215/S0012-7094-04-12432-7},
      review={\MR{2084613}},
}

\bib{GromovPiatetski-Shapiro}{article}{
      author={Gromov, M.},
      author={Piatetski-Shapiro, I.},
       title={Nonarithmetic groups in {L}obachevsky spaces},
        date={1988},
        ISSN={0073-8301,1618-1913},
     journal={Inst. Hautes \'Etudes Sci. Publ. Math.},
      number={66},
       pages={93\ndash 103},
         url={http://www.numdam.org/numdam-bin/item?id=PMIHES_1987__66__93_0},
      review={\MR{932135}},
}

\bib{HuglandConvex}{article}{
      author={Haglund, Fr\'ed\'eric},
       title={Finite index subgroups of graph products},
        date={2008},
        ISSN={0046-5755,1572-9168},
     journal={Geom. Dedicata},
      volume={135},
       pages={167\ndash 209},
         url={https://doi.org/10.1007/s10711-008-9270-0},
      review={\MR{2413337}},
}

\bib{Klein}{article}{
      author={Klein, Felix},
       title={Neue {B}eitr\"age zur {R}iemann'schen {F}unctionentheorie},
        date={1883},
        ISSN={0025-5831,1432-1807},
     journal={Math. Ann.},
      volume={21},
      number={2},
       pages={141\ndash 218},
         url={https://doi.org/10.1007/BF01442920},
      review={\MR{1510193}},
}

\bib{Lehmer}{article}{
      author={Lehmer, D.~H.},
       title={Factorization of certain cyclotomic functions},
        date={1933},
        ISSN={0003-486X,1939-8980},
     journal={Ann. of Math. (2)},
      volume={34},
      number={3},
       pages={461\ndash 479},
         url={https://doi.org/10.2307/1968172},
      review={\MR{1503118}},
}

\bib{LongTotGeod}{article}{
      author={Long, D.~D.},
       title={Immersions and embeddings of totally geodesic surfaces},
        date={1987},
        ISSN={0024-6093,1469-2120},
     journal={Bull. London Math. Soc.},
      volume={19},
      number={5},
       pages={481\ndash 484},
         url={https://doi.org/10.1112/blms/19.5.481},
      review={\MR{898729}},
}

\bib{LiOhshikaWang}{article}{
      author={Li, Liulan},
      author={Ohshika, Ken'ichi},
      author={Wang, Xiantao},
       title={On {K}lein-{M}askit combination theorem in space {II}},
        date={2015},
        ISSN={0386-5991,1881-5472},
     journal={Kodai Math. J.},
      volume={38},
      number={1},
       pages={1\ndash 22},
         url={https://doi.org/10.2996/kmj/1426684440},
      review={\MR{3323511}},
}

\bib{McMullen}{article}{
      author={McMullen, Curtis~T.},
       title={Hausdorff dimension and conformal dynamics. {I}. {S}trong
  convergence of {K}leinian groups},
        date={1999},
        ISSN={0022-040X,1945-743X},
     journal={J. Differential Geom.},
      volume={51},
      number={3},
       pages={471\ndash 515},
         url={http://projecteuclid.org/euclid.jdg/1214425139},
      review={\MR{1726737}},
}

\bib{MatsuzakiTaniguchi}{book}{
      author={Matsuzaki, Katsuhiko},
      author={Taniguchi, Masahiko},
       title={Hyperbolic manifolds and {K}leinian groups},
      series={Oxford Mathematical Monographs},
   publisher={The Clarendon Press, Oxford University Press, New York},
        date={1998},
        ISBN={0-19-850062-9},
        note={Oxford Science Publications},
      review={\MR{1638795}},
}

\bib{PurcellSouto}{article}{
      author={Purcell, Jessica~S.},
      author={Souto, Juan},
       title={Geometric limits of knot complements},
        date={2010},
        ISSN={1753-8416,1753-8424},
     journal={J. Topol.},
      volume={3},
      number={4},
       pages={759\ndash 785},
         url={https://doi.org/10.1112/jtopol/jtq020},
      review={\MR{2746337}},
}

\bib{SalemNo}{article}{
      author={Salem, R.},
       title={Power series with integral coefficients},
        date={1945},
        ISSN={0012-7094,1547-7398},
     journal={Duke Math. J.},
      volume={12},
       pages={153\ndash 172},
         url={http://projecteuclid.org/euclid.dmj/1077472968},
      review={\MR{11720}},
}

\bib{Scott}{article}{
      author={Scott, Peter},
       title={Subgroups of surface groups are almost geometric},
        date={1978},
        ISSN={0024-6107,1469-7750},
     journal={J. London Math. Soc. (2)},
      volume={17},
      number={3},
       pages={555\ndash 565},
         url={https://doi.org/10.1112/jlms/s2-17.3.555},
      review={\MR{494062}},
}

\bib{Serre}{book}{
      author={Serre, J.-P.},
       title={A course in arithmetic},
      series={Graduate Texts in Mathematics},
   publisher={Springer-Verlag, New York-Heidelberg},
        date={1973},
      volume={No. 7},
        note={Translated from the French},
      review={\MR{344216}},
}

\bib{Thomson2016}{article}{
      author={Thomson, Scott},
       title={Quasi-arithmeticity of lattices in {$\mathrm{PO}(n,1)$}},
        date={2016},
        ISSN={0046-5755,1572-9168},
     journal={Geom. Dedicata},
      volume={180},
       pages={85\ndash 94},
         url={https://doi.org/10.1007/s10711-015-0092-6},
      review={\MR{3451458}},
}

\bib{ThurstonNotes}{misc}{
      author={Thurston, William},
       title={The geometry and topology of 3-manifolds},
        type={Lecture Notes},
 institution={Princeton University Math. Dept},
        date={1978},
}

\bib{Wise2021}{book}{
      author={Wise, Daniel~T.},
       title={The structure of groups with a quasiconvex hierarchy},
      series={Annals of Mathematics Studies},
   publisher={Princeton University Press, Princeton, NJ},
        date={2021},
      volume={209},
        ISBN={[9780691170442]; [9780691170459]; [9780691213507]},
      review={\MR{4298722}},
}

\bib{Zevenbergen}{article}{
      author={Zevenbergen, Matthew},
       title={Connectivity in the space of framed hyperbolic 3-manifolds},
        date={2024},
     journal={arXiv preprint},
      eprint={arXiv:2410.15537},
         url={https://arxiv.org/abs/2410.15537},
}

\end{biblist}
\end{bibdiv}

\end{document}